\documentclass[oneside, 10pt]{amsart}

\usepackage[letterpaper,left=1.2in,right=1.2in,top=1in,bottom=1in]{geometry}

\title{The prime divisors of the period and index of a Brauer class}
\author{Benjamin Antieau and Ben Williams}

\usepackage{palatino, euler}
\usepackage{setspace}

\usepackage{amsrefs}

\usepackage[pdfstartview=FitH]{hyperref}
\usepackage[english]{babel}
\usepackage{geometry}
\usepackage[utf8x]{inputenc}
\usepackage{amsmath, amssymb, amsthm}
\usepackage{graphicx}
\usepackage[colorinlistoftodos]{todonotes}
\usepackage{xy}
\xyoption{all}

\usepackage{color}
\definecolor{td}{rgb}{1,0,0}

\definecolor{ex}{rgb}{0,0.2,0.5}

\newtheorem{proposition}{Proposition}
\newtheorem{theorem}[proposition]{Theorem}
\newtheorem{lemma}[proposition]{Lemma}

\newtheorem*{theoremu}{Theorem}

\newcommand{\GL}{\operatorname{GL}}
\newcommand{\SL}{\operatorname{SL}}
\newcommand{\Gm}{\mathbb{G}_m}
\newcommand{\PGL}{\operatorname{PGL}}
\newcommand{\PSL}{\operatorname{PSL}}
\newcommand{\Hoh}{\mathrm{H}}
\newcommand{\Br}{\operatorname{Br}}
\newcommand{\ind}{\operatorname{ind}}
\newcommand{\per}{\operatorname{per}}
\providecommand{\deg}{\operatorname{deg}}
\newcommand{\Aut}{\operatorname{Aut}}
\newcommand{\iso}{\cong}
\newcommand{\ZZ}{\mathbb{Z}}
\newcommand{\tensor}{\otimes}

\newcommand{\Supp}{\operatorname{Supp}}
\newcommand{\Hom}{\operatorname{Hom}}
\newcommand{\End}{\operatorname{End}}
\newcommand{\diag}{\operatorname{diag}}
\newcommand{\Spec}{\operatorname{Spec}}
\newcommand{\Mat}{\operatorname{Mat}}
\newcommand{\et}{\textit{\'et}}

\newcommand{\cat}[1]{\mathbf{#1}}

\newcommand{\extended}[1]{}
\newcommand{\concise}[1]{#1}
\newcommand{\noconcise}[1]{#1} 

\begin{document}

\begin{abstract}
    We show that in locally-ringed connected topoi the primes dividing the period and index of a Brauer class coincide.
    The result applies in particular to Brauer classes on connected schemes, algebraic stacks, topological spaces and to
    the projective representation theory of profinite groups.
\end{abstract}
\maketitle

\section{Introduction}

If $k$ is a field, then the Brauer group, $\Br(k)$, is the group of equivalence classes of central simple $k$--algebras modulo Morita equivalence. A theorem of
Wedderburn's states that every central simple $k$--algebra $A$ is isomorphic to a matrix algebra $\Mat_n(D)$, where $D$ is a finite-dimensional central $k$--division algebra. Since the rings $D$ and $\Mat_n(D)$ are Morita equivalent, the Brauer group is identified with
the set of isomorphism classes of finite-dimensional $k$--division-algebras.  The Brauer group was introduced by Brauer in the 1930s, and
has been studied extensively since\extended{---the monograph of Gille and Szamuely~\cite{gille2006} is a good reference}.

For a class $\alpha \in \Br(k)$, one defines the period $\per(\alpha)$ to be its order as a group element. If $A$ is a
central simple $k$--algebra, write $\alpha = [A]$ for the associated class in the Brauer group; the integer
$\per(\alpha)$ is the smallest positive integer such that $A^{\otimes_k \per(\alpha)}\iso\Mat_n(k)$ for some integer
$n$. The index, $\ind(\alpha)$, is the greatest common divisor of the degrees---the square-roots of the dimensions over
$k$---of the central simple algebras in the class $\alpha$.

It is not hard to show that $\per(\alpha)|\ind(\alpha)$.
Three additional facts about $\per(\alpha)$ and $\ind(\alpha)$ concern us in this paper, all of
which are classical and can be found in~\cite{gille2006}:
\begin{enumerate}
\item \label{1} $\ind(\alpha)$ is the degree of the lowest-dimensional element of $\alpha$, namely the unique division
  algebra $D$ with $[D]=\alpha$;
\item \label{2} one may use Galois splitting fields and Sylow subgroups of Galois groups to prove that $\per(\alpha)$ and
  $\ind(\alpha)$ have the same prime divisors;
\item \label{3} as a consequence, there exists a central simple algebra $A$, specifically the unique division algebra $D$ in
  the class $\alpha$, with class $\alpha$ such that $\deg(A)$ has the same prime divisors as $\per(\alpha)$.
\end{enumerate}

Work of Azumaya~\cite{azumaya1951} and then of Auslander and Goldman~\cite{auslander1960} established the notion of
an \textit{Azumaya algebra} over a commutative ring $R$, and defined $\Br(R)$ as a group of equivalence classes of
Azumaya algebras, generalizing the Brauer group of a field. These Azumaya algebras are flat families of central simple
algebras.  The idea was extended by Grothendieck~\cite{grothendieck1968-a} to the case of a locally-ringed topos,
$(\cat{X},R)$, although the emphasis in that work was on the specific case where the topos is $\tilde X_{\text{\'et}}$,
the \'etale topos of a scheme $X$, and where the local ring is $\mathscr{O}_X$, the structure sheaf of $X$.\extended{ The
definition of Azumaya algebras and $\Br(\cat{X}, R)$, when applied to the \'etale topos of $\Spec R$, locally ringed by
the structure sheaf, specializes to the definitions of Auslander and Goldman.}

In the generality of a locally ringed topos, it is possible to define the period and the index of a class $\alpha \in \Br(\cat{X}, R)$, although
one must allow for pathologies if $\cat{X}$ is badly disconnected.\extended{ For instance $\Br(\cat{X},R)$ need not be a torsion
group in general, so $\per(\alpha)$ may be infinite.} Unless otherwise stated, we assume $\cat{X}$ is connected\extended{, an
assumption which greatly simplifies the theory and costs very little applicability}.  We define $\per(\alpha)$ as the
order of $\alpha$, which is finite under this assumption, and we may define
\begin{equation*}
    \ind(\alpha)=\gcd\left\{\deg(A) :[A]=\alpha\right\}.
\end{equation*}
We wish to determine whether the statement $\per(\alpha)| \ind(\alpha)$ and the analogues of \eqref{1}--\eqref{3} above
hold in general.

Of these, $\per(\alpha) | \ind(\alpha)$ is generally seen to be true, whereas we have already proved in~\cite{antieau????} that
\eqref{1} does not always hold. \extended{Namely, we showed that there exists a smooth affine complex $6$-fold $X$ and a
Brauer class $\alpha\in\Br(X)$ with $\per(\alpha)=2$ and $\ind(\alpha)=2$, but where there is no degree $2$ Azumaya
algebra defined on $X$ with class $\alpha$. In more recent work,  \cite{antieau2013}, we extended the arguments of~\cite{antieau????} to
give examples of this failure with $\per(\alpha)=p$ for every prime $p$. Our smallest example is $6$-dimensional. In a
positive direction, it is known that if $X$ is a regular noetherian $2$-dimensional scheme, then \eqref{1} holds: there
exists an Azumaya algebra of degree equal to $\per(\alpha)$. In the affine case, this follows from the proof
of~\cite{auslander1960}*{Proposition 7.4}, and the general case is similar.}

Property \eqref{3} was known to hold in the following situations: the classical case of the \'etale sites of fields, and
the \'etale sites of regular noetherian $2$--dimensional schemes, \cite{auslander1960}, and the \'etale sites of
schemes $X$ that are unions of two affine schemes along an affine intersection, where it was deduced by~\cite{gabber1981}*{Chapter
  II}. To our knowledge, no other results along the lines of \eqref{3} were known for schemes.

If property \eqref{3} obtains, then property \eqref{2} must obtain as well.\extended{ We asked in~\cite{antieau2011}*{Problem 1.8} when
property \eqref{2} holds and we are aware of some additional cases where it was known to hold where \eqref{3} was not
known.} By~\cite{antieau2011}*{Theorem 3.1}, \eqref{2} holds for finite CW complexes. The proof employs the Hurewicz isomorphism
theorem and twisted topological $K$--theory, and is peculiar to the case of CW complexes. If $X$ is a regular noetherian scheme, then we showed
in~\cite{antieau2013}*{Proposition 6.5} that the period and index have the same prime divisors. This time, the proof was by
using the inclusion $\Br(X)\subseteq\Br(K)$, where $K$ is the field of fractions of $X$, and an argument of Saltman to
show that the index of $\alpha\in\Br(X)$ is the same when computed over $X$ or over $K$.


\extended{In this paper, we consider the common generalization of Azumaya algebras in the \'etale topology on a scheme and in the
ordinary topology on a CW complex: the theory of Azumaya algebras in a locally ringed topos.}

Azumaya algebras in locally ringed topoi also generalize the special case of projective representations of finite
groups. Given a finite group, $G$, we may
form the topos $\cat{B}G$ of discrete $G$--sets, and endow it with a local ring $R$. \extended{In this topos, an Azumaya algebra of degree $n$ is
tantamount to a representation $\rho: G \to \PGL_n(R)$.} When $R = \mathbb{C}$, the Brauer group of the locally ringed
topos $\Br(\cat{B}G, R)$ is the Schur multiplier $\Hoh^2(G,\mathbb{C}^\times)$ of $G$.
In this setting, \eqref{1} is known not to hold. Higgs communicated to us that
$\mathrm{PSL}_2(\mathbb{F}_7)$ has Brauer group $\ZZ/2$, where the non-zero class, $\alpha$, is represented
by irreducible projective representations of degrees $4, 4, 6, 6, 8$, so that
$\per(\alpha)=\ind(\alpha)=2$, but there is no degree-$2$ Azumaya algebra with class
$\alpha$. On the other hand, \eqref{2} is known to hold for all finite $G$, and
\eqref{3} is known to hold when $G$ is a finite $p$--group, \cite{higgs1988}.

In this paper, we establish \eqref{3}, and therefore \eqref{2}, under mild hypotheses on the topos. This solves Problem
1.8 of~\cite{antieau2011}, where to our knowledge the question of whether \eqref{2} holds in full generality was first
posed.  \extended{
\begin{theoremu}[Theorem~\ref{th:main}]
  Let $(\cat{X},R)$ be a connected locally ringed topos, and let $\alpha\in\Br(\cat{X},R)$. There exists an Azumaya algebra over
  $(\cat{X},R)$ the degree of which is divisible only by those primes dividing $\per(\alpha)$. In particular, $\per(\alpha)$ and
  $\ind(\alpha)$ have the same prime divisors.
\end{theoremu}}

\extended{While the result is stated in a general and abstract language, the proof when it comes is simple, being little more than the
construction of homomorphisms between projective linear groups. It would be possible to give
a different, but conceptually identical, proof in the language of twisted sheaves.}

The paper therefore provides, in the first place, a unified proof of a statement that had previously been
known only by different arguments in different contexts. In the second place, it covers the cases of the \'etale site on
singular or non-noetherian schemes and the case of infinite CW complexes. In the third, it strengthens the result of
Saltman for regular noetherian schemes by producing an Azumaya algebra the degree of which is divisible only by primes
dividing the period.

The main theorem, Theorem \ref{th:main}, is not stated in maximum generality; it holds for instance if the hypothesis that $\cat{X}$ be
connected is weakened to the hypothesis that $\pi_0(\cat{X})$ be compact. In
\cite{gabber1981}, a more general
definition of Azumaya algebra than that of \cite{grothendieck1968-a} is given, appertaining to the case of a ringed
topos. The two definitions coincide in the cases of \'etale sites of schemes and in the case of CW complexes locally
ringed by the sheaf of continuous complex-valued functions. We do not explore an expansion of Theorem \ref{th:main}
to the generality of the Azumaya algebras of \cite{gabber1981}.

We would like to thank the referee, whose advice greatly improved the presentation.

\section{Azumaya Algebras in Grothendieck Topoi}

\noconcise{
In this section and the next, we present the theories of Azumaya algebras and of
$\PGL_n$--bundles in a locally-ringed connected Grothendieck topos, and show that they are equivalent.
Our reference for the theory of topoi is \cite{artin1972-b}, and we adopt the theory of universes of \cite{artin1972-b}*{Expos\'e I,
Appendice}. Chapter 0 of \cite{giraud1971-a} is an expedited guide to the set- and topos-theoretic preliminaries.
In general, the topoi we encounter shall all have enough points, although this is inessential to the argument.
The exposition is greatly
simplified if we assume $\cat{X}$ is connected, which is also the most applicable case, so all topoi we consider are
assumed connected.

There is an abelian category of $R$--modules, in which one may form free- and locally-free-modules, tensor products, and
homomorphism objects.
\extended{If $U$ is an object of $\cat{X}$ and $f \in R(U)$, we define $U_f$ to be the largest subobject of $U$ in which the image of $f$ is
invertible; it is the pull-back:
\[ \xymatrix{ U_f \ar[r] \ar[d] & R^\times \ar^u[d] \\ U \ar^f[r] & R.} \] A ringed topos, $(\cat{X}, R)$, is
\textit{locally-ringed} if the ring object $R$ is commutative for any object $U$ in
$\cat{X}$, and for any $f \in R(U)$, the
objects $U_f$ and $U_{1-f}$ cover $U$---see}\concise{There is a notion of local ring object, see} \cite{artin1972-b}*{Expos\'e IV, Exercice 13.9} or \cite{grothendieck1968-a}*{Section
2} and \cite{maclane1992-a}*{Chapter VIII}. \extended{If $p$ is a point of $X$ and $R$ is a local ring object, then
the ring $p^*(R)$ has the property that for all elements $f \in p^*(R)$, either $f$ or $(1-f)$ is a unit. Consequently,
either $p^*R$ is empty or $p^*R$ has a unique maximal ideal.} In the presence of enough points of the topos, a ring object
is a local ring object if and only if $p^*R$ is either local or empty at all points.
Given a ring object $R$, assumed throughout to be unital and associative, we may form the group $R^\times$ of units in $R$\extended{ as the
limit
\[ \xymatrix{ R^\times \ar[r] \ar[d] & \ast \ar^1[d] \\ R \times R \ar^\mu[r] & R. } \]
There is a composite morphism $u: R^\times \to R \times R \overset{\pi_1}{\to} R$ given by projection on the first factor. Since the definition
of $R^\times$ is as a limit, and the formation of limits commutes with the Yoneda embedding, it follows that for all objects $U$ of
$\cat{X}$, the map of sets $u: R^\times(U)  \to R(U)$ is an injection with image $R(U)^\times$. In particular, $u: R^\times \to R$ is a submonoid of the
multiplicative monoid structure on $R$, and arguing with elements of $R^\times(U)$ we see that $R^\times$ is a group}.

We recall from \cite{grothendieck1968-a} that an \textit{Azumaya algebra}, $\mathscr{A}$, on $(\cat{X}, R)$ is an
$R$-algebra in $\cat{X}$ which locally is isomorphic to an algebra of the form $\Mat_n(R)$; the integer $n$ is called the \textit{degree} of $\mathscr{A}$. The tensor product
$\mathscr{A} \tensor_R \mathscr{A}'$ is an Azumaya $R$-algebra formed by means of the Kronecker product $\Mat_n(R) \tensor \Mat_m(R)$, applied locally.

\extended{A \textit{locally free module of finite rank} on $(\cat{X}, R)$ is an $R$--module that locally is
isomorphic to $R^n$ where $n$ is an integer called the \textit{rank}.}

The \textit{Brauer group} $\Br(\cat{X},R)$ of $(\cat{X}, R)$ is the set of Azumaya algebras under the equivalence
relation that says
$\mathscr{A} \simeq \mathscr{A}'$ if there exist locally free $R$--modules $E$ and $E'$ of finite rank such that
\[ \mathscr{A} \tensor_R \End_R(E) \iso \mathscr{A}' \tensor_R \End_R(E').\]
The Brauer group is indeed a group under tensor product, with the inverse of $\mathscr{A}$ being given by the opposite
algebra\extended{, since $\mathscr{A} \tensor_R \mathscr{A}^{\text{op}} \iso \End_R (\mathscr{A})$}.

If $V$ is a free $R$--module then the exterior power $V^{\wedge d}$ \concise{may be defined in the evident way.}
In a ringed topos, it is possible to define the group objects $\GL_n$, $\Gm \iso \GL_1$, $\SL_n$, $\PGL_n \iso \GL_n /
\Gm$\extended{ and $\PSL_n \iso \SL_n/\mu_n$. If the $n$--th power map $\Gm \to \Gm$ is an epimorphism, as often happens, then
$\PGL_n \iso \PSL_n$}.

If $\mathscr{A}$ is an Azumaya algebra, then it is possible to form $\Aut(\mathscr{A})$ as a group in $\cat{X}$; locally
this group is isomorphic to a group of the form $\Aut(\Mat_n(R))$. The conjugation action of $GL_n$ on $\Mat_n$ means
that there is a homomorphism $\phi: \PGL_n \to \Aut(\Mat_n(R))$.

The following proposition is asserted in \cite{grothendieck1968-a}.
}

\begin{proposition}
  If $(\cat{X}, R)$ is a locally-ringed topos, then $\phi:\PGL_n\to\Aut(\Mat_n(R))$ is an isomorphism.
\end{proposition}
\begin{proof}
  We refer to \cite{maclane1992-a}*{Chapter VIII, Theorem 3}, which says that there is a universal locally ringed
  topos. It is $\Spec(\ZZ, \mathcal{O})$, the ringed topos associated to the Zariski site on $\Spec \ZZ$. Given any locally-ringed topos $(\cat{X}, R)$, there is a
  geometric morphism $r: \cat{X} \to \Spec \ZZ$ such that $r^* \mathcal{O} \iso R$. Since $r^*$ preserves finite limits
  and all colimits, it follows that $R^n \iso (r^* \mathcal{O})^n$, that $\End(R^n) \iso r^* \End( \mathcal{O} )$, that
  $\PGL_n(R) \iso r^* \PGL_n(\mathcal{O})$, that $\Aut(\End(R^n)) \iso r^*\Aut\left(\End(\mathcal{O}^n)\right)$, and all
  these isomorphisms are compatible with the various actions of these objects on themselves and each other.

  It suffices, therefore, to prove the proposition in the case of $(\Spec \ZZ, \mathcal{O})$. Since every projective
  $\ZZ$--module is free, the result follows from Theorem 3.6 and Proposition 5.1 of
  \cite{auslander1960}, the Skolem-Noether theorem.
%
%
%
\end{proof}

\extended{
The proof in the case where $\cat{X}$ has enough points may be carried at stalks, using the same results in
\cite{auslander1960} as were used for $\Spec \ZZ$. In the absence of enough points, arguing in $\cat{X}$ would require
some ungainly maneuvering among the objects of $\cat{X}$ in order to mimic the properties of a local ring at a stalk. It
is to avoid this that we employ the artifice of the universal example.

This proposition is the point at which it becomes necessary for the topos to be locally ringed, rather than merely
ringed, which is why we draw attention in the statement to our standing assumption that $(\cat{X}, R)$ is locally
ringed.}

There exist cohomology functors $\Hoh^i(G)$ in the topos $\cat X$, defined for $i=0,1$ in the case of a
nonabelian group $G$, but for all $i\ge 0$ in the case of an abelian group $A$ (see \cite{artin1972}*{Expos\'e V} and
\cite{giraud1971-a}).\extended{ The set $\Hoh^1(G)$ classifies $G$--torsors in the topos, and since $\Aut(\Mat_n(R)) \iso \PGL_n$,
there is a natural bijection between isomorphism classes of Azumaya algebras of degree $n$ and $\Hoh^1(\PGL_n)$. We view
elements $\mathscr{A} \in \Hoh^1(\PGL_n)$ as being Azumaya algebras of degree $n$. It is a consequence of the
properties of Kronecker product and matrix multiplication that the two definitions of $\tensor$ on the classes in
$\Hoh^1(\PGL_*)$, one given by tensor product of $R$--algebras, the other by Kronecker products of matrices,
agree.}
\extended{Associated to a short exact sequence of groups \[ 1 \to G \to G'' \to G' \to 1 \] there is a long exact sequence in
cohomology, extending to $\Hoh^2(G)$ in the case of a central extension by an abelian group $G$, so that we have, in
particular, the portion of a long exact sequence}\concise{The short exact sequence of groups
$1\rightarrow\Gm\rightarrow\GL_n\rightarrow\PGL_n\rightarrow 1$ yields a portion of an
exact sequence in nonabelian cohomology}
\begin{equation}
  \label{eq:5}
   \Hoh^1(\GL_n) \to \Hoh^1(\PGL_n) \overset{\delta}{\to} \Hoh^2(\Gm).
\end{equation}
The map $\Hoh^1(\GL_n) \to \Hoh^1(\PGL_n)$ takes a locally free $R$--module $E$ of rank $n$ to the $\PGL_n$--torsor
$\End_R(E)$.\extended{

If $G$ and $G'$ are two groups, then there is an isomorphism $\Hoh^i(G \times G') \iso \Hoh^i(G) \times \Hoh^i(G')$,
where applicable, by Giraud~\cite{giraud1971-a}*{Remarque 2.4.4}.
This endows the cohomology of an abelian group $A$ in $\cat{X}$ with an abelian group structure, and
does so in such a way that the $n$-th power map $A \to A$, which on the level of elements is $a \mapsto a^n$ if $A$ is
written multiplicatively, induces multiplication by $n$ on the additive abelian groups
$\Hoh^*(A)$.

Writing $\tensor$ for the Kronecker product $\GL_n \tensor \GL_m \to \GL_{nm}$, and for the induced product $\PGL_n \tensor
\PGL_m \to \PGL_{nm}$,we have a commutative diagram of short exact sequences of groups:
\begin{equation*}
  \xymatrix{ 1 \ar[r] & \Gm \times \Gm \ar[r] \ar_{\text{mult}}[d] & \GL_n \times \GL_m \ar[r] \ar_{\tensor}[d] &
    \PGL_n \times \PGL_m \ar[r] \ar_{\tensor}[d] & 1 \\ 1 \ar[r] & \Gm \ar[r] & \GL_{nm} \ar[r] & \PGL_{nm} \ar[r] & 1.} 
\end{equation*}
In particular, this means that }
\concise{We have}
\begin{equation\concise{*}}
    \extended{\label{eq:4}}
  \delta(\mathscr{A} \tensor \mathscr{A}') = \delta(\mathscr{A}) + \delta(\mathscr{A}').
\end{equation\concise{*}}

The following proposition holds in general, the proof being the same is in the case of the \'etale topos of a
scheme. It is also implicit in the discussion of \cite{giraud1971-a}*{Chapitre V.4.4--5}.

\begin{proposition} \label{p:BrtoH2Gm}
  If $(\cat{X}, R)$ is a connected, nonempty, locally-ringed topos, then $\Br(\cat X, R)$
  can be identified with the image of the map
  \begin{equation\concise{*}}
      \extended{\label{eq:1}}\coprod_{n=1}^\infty \Hoh^1 (\PGL_n) \to \Hoh^2(\Gm).
    \end{equation\concise{*}}
\end{proposition}

\medskip
\extended{
\subsection{Aside on Topoi that are not Connected} \label{ss:disconnect2}

In general, a proposition similar to Proposition \ref{p:BrtoH2Gm} holds, but where the objects $\PGL_n$ are replaced by objects
$\PGL_{c}$ where $c$ is a global section of the constant sheaf $\tilde \ZZ_{\ge 0}$.  This allows us to identify $\Br(\cat{X}, R)$ with
a subgroup of $\Hoh^2(\Gm)$ in all cases.

If every locally free $R$--module $E$ of locally constant rank can be extended to a locally free module $E \oplus R^c$ of constant rank,
where $c$ is a global section of the constant sheaf $\tilde \ZZ_{\ge 0}$, and if every Azumaya algebra $\mathscr{A}$ can be extended to an
Azumaya algebra $\mathscr{A} \tensor \Mat_{c'}(R)$ of constant degree, again where $c'$ is a global section of the constant sheaf $\tilde
\ZZ_{\ge 0}$, then $\Br(\cat X, R)$ agrees with the image of the map \eqref{eq:1} as written. This is the case if all global sections of
$\tilde \ZZ_{\ge 0}$, i.e.,~all maps $\ast \to \tilde \ZZ_{\ge 0}$, factor through some map $\ast \to \tilde n$, where $\tilde n$ denotes the
constant sheaf associated to $\{0, 1, \dots, n\}$. Such a factorization is guaranteed if the pro-set $\pi_0(\cat X)$ is compact.
}

\section{Period \& Index}
Henceforth we assume our topos locally-ringed and connected.

Suppose $\alpha$ is an element of $\Br(\cat X,R) \subset \Hoh^2(\Gm)$. We define the \textit{period} of $\alpha$ to be the
order of $\alpha$ as a group element in $\Hoh^2(\Gm)$, assuming it is finite, and we define the \textit{index} of $\alpha$ to be the
greatest common divisor of all integers $n$ such that $\alpha$ is in the image of $\Hoh^1(\PGL_n)
\to \Hoh^2(\Gm)$. If $\mathscr{A}$ is an element in $\Hoh^1(\PGL_n)$, then we say that the
\textit{degree} of $\mathscr{A}$ is $n$, and we abuse terminology in saying that the period and
index of $\mathscr{A}$ are simply the period and index of the image of $\mathscr{A}$ under the map
\[ \Hoh^1(\PGL_n) \to \Br(\cat{X},R).\] Writing $\alpha$ for this image, we say $\mathscr{A}$
\textit{represents} $\alpha$.

\begin{theorem} \label{p1} If $(\cat X,R)$ is a ringed topos and if $\alpha \in \Br(\cat X,R)$ is an element in the Brauer group
  represented by $\mathscr{A}$, then $\per(\alpha)$ divides the degree of $\mathscr{A}$. Consequently,
  $\per(\alpha)|\ind(\alpha)$, and $\Br(\cat X,R)$ is  torsion.
\end{theorem} 
\extended{\begin{proof}
  By reference to diagram \eqref{diagram:PGPS}, and noting that $\epsilon_n$ induces multiplication by $n$ in
  cohomology, we see that thereThere is a factorization of the composite map 
\[ \xymatrix@C=40pt{\Hoh^1(\PGL_n)
  \ar^{\iota}[r] & \Hoh^2(\Gm) \ar^{\times n}[r] & \Hoh^2(\Gm)} \] 
as \[ \xymatrix{\Hoh^1(\PGL_n) \ar[d] \\
    \Hoh^1(\nu_n) \ar[r] & \Hoh^2(\Gm/\mu_n) \ar[r] & \Hoh^2(\Gm),} \] which is necessarily the trivial map. As a
  consequence, the image of $\iota: \Hoh^1(\PGL_n) \to \Hoh^2(\Gm)$ is $n$--torsion, and it follows that $\per(\alpha) |
  \deg(\mathscr{A})$ if $\mathscr{A}$ is any representative of $\alpha$. Since $\ind(\alpha)$ is the greatest common
  divisor of all such $\mathscr{A}$, the result follows. 
\end{proof}}

\concise{The proof is the same as in the case of the \'etale site of a scheme, where it is
standard. For a proof in the language of gerbes, see \cite{giraud1971-a}*{Chapitre V.4.6}.}

\begin{proposition} \label{pr:summation}
  Suppose $\mathscr{A}$ in $\Hoh^1(\PGL_n)$  and $\mathscr{A}'$ in $\Hoh^1(\PGL_m)$ each represent the same element, $\alpha$, in
  $\Br(\cat{X}, R)$. There exists an element $\mathscr{A} \oplus \mathscr{A}'$ in $\Hoh^1(\PGL_{n+m})$ representing $\alpha$.
\end{proposition}

\begin{proof}
  Suppose $\mathscr{A}$ in $\Hoh^1(\PGL_n)$ and $\mathscr{B}$ in $\Hoh^1(\PGL_m)$ represent $\alpha$ and $\beta$ in
  $\Br(\cat{X},R)$, respectively.
  \extended{As remarked above, there}\concise{There} is an isomorphism $\Hoh^1(\PGL_n \times \PGL_m) \to \Hoh^1(\PGL_n) \times \Hoh^1(\PGL_m)$, \cite{giraud1971-a}. The data of $\mathscr{A}$ and $\mathscr{B}$ therefore give an element $\mathscr{A} \times
  \mathscr{B}$ in $\Hoh^1(\PGL_n \times \PGL_m)$.
  
  We may include $\Gm\xrightarrow{\Delta} \GL_n \times \GL_m$ as the subgroup of scalar matrices, and we define $\PGL_{n,m}$ as
  the quotient. We also write $\Delta: \Gm \to \Gm \times \Gm$ for the diagonal inclusion. There is a short exact
  sequence of group objects
  \[ \xymatrix{ 1 \ar[r] & \Gm \ar^<<<<{\Delta}[r] & \Gm \times \Gm \ar^<<<<{q}[r] & \Gm \ar[r] & 1 } \]
   where the map $q: \Gm \times \Gm \to \Gm$ is that given by $(\lambda, \lambda') \mapsto \lambda/\lambda'$; this map is
  split, and consequently an epimorphism.

  The rows and first two columns of the following diagram are short exact sequences of groups:
  \[ \xymatrix{ & 1 \ar[d] & 1 \ar[d] & 1 \ar[d] \\ 1 \ar[r] &\Gm \ar^<<<<<<{\Delta}[r] \ar@{=}[d] & \Gm \times \Gm \ar[r] \ar[d] & \Gm
  \ar[d] \ar[r] & 1 \\
1 \ar[r] & \Gm \ar^<<<<<<\Delta[r] \ar[d] & \GL_n \times \GL_m \ar[r] \ar[d] & \PGL_{n,m} \ar[d] \ar[r] & 1\\
1 \ar[r] & 1 \ar[d] \ar[r] & \PGL_n \times \PGL_n \ar@{=}[r] \ar[d] & \PGL_n \times \PGL_m \ar[r] \ar[d] & 1 \\ & 1 & 1 & 1.} \]
By the nine-lemma, the third column is also exact. We conclude that the obstruction to lifting $\mathscr{A} \times
\mathscr{B}$ from $\Hoh^1(\PGL_n \times \PGL_m)$ to $\Hoh^1(\PGL_{n,m})$ is the class $\alpha - \beta$ in $\Hoh^2(\Gm)$.

If we take $\mathcal{A}$ and $\mathcal{A}'$, both of which represent $\alpha \in \Br(\cat{X}, R)$, then this
obstruction vanishes, and we may define $\mathscr{A}''$ to be a lift of $\mathscr{A} \times \mathscr{A}'$ to
$\Hoh^1(\PGL_{n,m})$.

There is a `direct-summation' map $\sigma: \GL_n \times \GL_m \to \GL_{n+m}$. The diagram
\[ \xymatrix{ 1 \ar[r] & \Gm \ar@{=}[d] \ar^<<<<<\Delta[r] & \GL_n \times \GL_m \ar[r] \ar^\sigma[d] & \PGL_{n,m} \ar[d] \ar[r] & 1
  \\ 1 \ar[r] & \Gm \ar[r] & \GL_{n+m} \ar[r] & \PGL_{n+m} \ar[r] &  1 } \]
commutes, from which we deduce that $\mathscr{A}''$ yields an element $\mathscr{A}\boxplus\mathscr{A}'$ in
$\Hoh^1(\PGL_{n+m})$. This element represents $\alpha$, since $\mathscr{A}''$ does.
\end{proof}

We write $\Supp n$ for the set of prime numbers dividing an integer $n$.

\begin{lemma} \label{l:ntl} Let $m$ and $n$ be positive integers, with $m \mid n$. Then there exists
a set of integers $\{q_1, \dots, q_\ell\}$ with $1 \le q_i <n $ and $(q_i, m) = 1$ for all $i$ and
such that
\begin{equation}\label{eq:2}
\Supp \gcd\left\{ \binom{n}{q_1}, \dots, \binom{n}{q_\ell} \right\} = \Supp m.
\end{equation}
\end{lemma}

If $\{q_1, \dots, q_\ell\}$ is a set meeting the conditions of the lemma, and if $q_{\ell +1}$ is some number such that $1 \le
q_{\ell+1} < n $ and $(q_{\ell+1} , m) = 1$, then it follows from the proof that $\{q_1, \dots, q_{\ell}, q_{\ell+1}\}$ also meets the conditions of
the lemma. The lemma could therefore be stated as saying that the maximal set $\{ q \: : \: 1 \le q < n,\,(q,m) =1 \}$ satisfies
\eqref{eq:2}.

 \begin{proof} Suppose $a$ and $b$ are two positive integers and $p$ is a prime. Then the
value of $\binom{a}{b}$ in $\ZZ/(p\ZZ)$ is the coefficient of $x^b$ in the expansion of $(1+x)^a$
over that ring. If $a=cp^s$ for some integer $s$, then \[(1+x)^a = \Big((1+x)^{p^s}\Big)^c \equiv
(1+x^{p^s})^c \pmod{p} \] from which we deduce that $\binom{a}{b}\equiv 0 \pmod{p}$ unless $p^s$
divides $b$ as well, in which case $\binom{a}{b} \equiv \binom{c}{b/p^s} \pmod{p}$.

Let $\{p_1, \dots, p_\ell\}$ denote the set of primes dividing $n$ but not dividing $m$. Let $s_i$
denote the exponent of the largest power of $p_i$ dividing $n$. The binomial coefficient
$\binom{n}{p_i^{s_i}} \equiv \binom{n/p_i^{s_i}}{1} \not \equiv 0 \pmod{p}$, while, for any prime
$q$ dividing $m$, we have $\binom{n}{p_i^{s_i}} \equiv 0 \pmod{q}$. The set $\{1, p_1^{s_1},
\dots, p_\ell^{s_\ell}\}$ therefore satisfies the assertions of the lemma. \end{proof}

Now we come to our main theorem, where we show that~\eqref{3} from the introduction, and hence~\eqref{2}, holds for a broad class of
locally ringed topoi.

\begin{theorem} \label{th:main}
  Let $(X,R)$ be a locally-ringed connected topos and let $\alpha \in \Br(X,R)$. There exists a
  representative $\mathscr{A}$ of $\alpha$ such that the prime numbers dividing $\per(\alpha)$ and $\deg(\mathscr{A})$ coincide.
\end{theorem} 
\begin{proof}
  Write $m$ for the period of $\alpha$, which is finite and divides $\ind(\alpha)$ by Theorem \ref{p1}. By definition of
  the Brauer group, there exists some positive integer $n$ and some $\mathscr{B} \in \Hoh^1(\PGL_n)$ such that
  $\mathscr{B}$ represents $\alpha$. By Theorem \ref{p1}, we know that $m|n$.

  Let $V$ denote the standard free $R$--module of rank $n$. Let $\{q_1, \dots, q_\ell\}$ be a set of integers $1 \le q_i < n$ with the
  properties that $(q_i, m) = 1$ for all $i$, and
\[ \Supp\gcd\left\{\binom{n}{q_1}, \dots, \binom{n}{q_\ell} \right\} =\Supp m.\] For each $i$, let $r_i$ be a positive
integer such that $q_ir_i \equiv 1 \pmod{m}$. For each $i$ between $1$ and $\ell$, define 
\[ W_i = \left(V^{\wedge {q_i}}\right)^{\tensor r_i}. \] 
The dimension of $W_i$ is \[ s_i = \binom{n}{q_i}^{r_i} \] and in particular 
\[\Supp \gcd \{s_1, \dots, s_\ell \} = \Supp m. \]

The formation of $W_i$ from $V$ means that there is a diagonal
homomorphism from $\GL(V) = \GL_n$ to $\GL(W_i) = \GL_{s_i}$, and this fits in
the following diagram \[ \xymatrix{ 1 \ar[r] & \Gm \ar_{z \mapsto z^{q_ir_i}}[d] \ar[r] & \GL_n  \ar[d] \ar[r] & \PGL_n \ar[d] \ar[r] & 1 \\
1 \ar[r] & \Gm \ar[r] & \GL_{s_i} \ar[r] & \PGL_{s_i} \ar[r] &1} \]
In particular, there is an induced map $f_i: \Hoh^1(\PGL_n) \to \Hoh^1(\PGL_{s_i})$ with the property that
$f_i(\mathcal{B})$ represents $q_ir_i\alpha = \alpha$ in $\Hoh^2(\Gm)$.

For any sufficiently large integer $g$ divisible by $\gcd \{ s_1 , \dots, s_\ell \}$ we can find
nonnegative integers $\{c'_1, \dots, c'_\ell \}$ such that \[ g = \sum_{i=1}^\ell c'_i s_i. \] In
particular, we can find some sufficiently large integer $N$ such that $\Supp N = \Supp m$ and \[ N =
\sum_{i=1}^\ell c_i s_i \] where the $c_i$ are nonnegative integers.

The elements $f_i(\mathscr{B})$ in $\Hoh^1(\PGL_{s_i})$ all represent $\alpha$, and by the construction of Proposition
\ref{pr:summation}, we can form 
\[ \mathscr{A} = \bigoplus_{i=1}^\ell \Big( \bigoplus_{j=1}^{c_i} f_i(\mathscr{B}) \Big), \] 
of degree $\deg(\mathscr{A}) = N$, which represents $\alpha$ as well. It lies in $\Hoh^1(\PGL_N)$. Since $\Supp N = \Supp m$, the theorem
is proved.
\end{proof}

We note that the bound on $\deg(\mathscr{A})$ implicit in the proof does not depend on the topos, and is probably
wildly inefficient in many interesting cases. For instance, in the case of an element $\alpha$ of
period $3$, represented by a class $\mathscr{A}$ of degree $60$, we must eliminate the primes $2$
and $5$. We may take as our set $\{q_1, q_2, q_3, q_4\} = \{ 1, 4, 55, 58 \}$, all of which are
congruent to $1$ modulo $3$, which means that we may take $r_1=r_2=r_3=r_4 = 1$. Setting 
\[ c_1= 137400, \quad c_2=1, \quad c_3=1, \quad c_4 = 88, \] and using the identity
\[N = 3^{15} = 137400 \binom{60}{1} + \binom{60}{4} +\binom{60}{55} + 88\binom{60}{58}, \]
we deduce that an element of $\Br(X,R)$ having period $3$ which is represented by $\mathscr{B}$ of degree $60$ may be
represented by an Azumaya algebra $\mathscr{A}$ of degree $3^{15}=14,348,907$.



\begin{bibdiv}
\begin{biblist}

\bib{antieau????}{article}{
      author={Antieau, Benjamin},
      author={Williams, Ben},
       title={Unramified division algebras do not always contain azumaya
  maximal orders},
        ISSN={0020-9910, 1432-1297},
     journal={Inventiones mathematicae},
     year={2013},
       pages={1\ndash 10},
         eprint={http://link.springer.com/article/10.1007/s00222-013-0479-7},
}

\bib{antieau2011}{article}{
      author={Antieau, Benjamin},
      author={Williams, Ben},
       title={The period-index problem for twisted topological $K$-theory},
        date={2014},
        no={2},
        volume={18}
     journal={Geometry \& Topology},
     pages={1115--1148}
}

\bib{antieau2013}{article}{
      author={Antieau, Benjamin},
      author={Williams, Ben},
       title={Topology and purity for torsors},
        date={2013},
     journal={ArXiv e-prints},
         eprint={http://arxiv.org/abs/1311.5273},
}

\bib{artin1972-b}{book}{
  editor={Artin, M.},
  editor={Grothendieck, A.},
  editor={Verdier, J.~L.},
  title={Th\'{e}orie des topos et cohomologie \'{e}tale des sch\'{e}mas.
  tome 1: Th\'{e}orie des topos},
      series={Lecture Notes in Mathematics, Vol. 269},
   publisher={{Springer-Verlag}},
     address={Berlin},
        date={1972},
        note={S\'{e}minaire de G\'{e}om\'{e}trie Alg\'{e}brique du {Bois-Marie}
  1963{\textendash}1964 {(SGA} 4), Dirig\'{e} par M. Artin, A. Grothendieck, et
  J. L. Verdier. Avec la collaboration de N. Bourbaki, P. Deligne et B.
  {Saint-Donat}},
}

\bib{artin1972}{book}{
      editor={Artin, M.},
      editor={Grothendieck, A.},
      editor={Verdier, J.~L.},
       title={Th\'{e}orie des topos et cohomologie \'{e}tale des sch\'{e}mas.
  tome 2},
      series={Lecture Notes in Mathematics, Vol. 270},
   publisher={{Springer-Verlag}},
     address={Berlin},
        date={1972},
        note={S\'{e}minaire de G\'{e}om\'{e}trie Alg\'{e}brique du {Bois-Marie}
  1963{\textendash}1964 {(SGA} 4), Dirig\'{e} par M. Artin, A. Grothendieck et
  J. L. Verdier. Avec la collaboration de N. Bourbaki, P. Deligne et B.
  {Saint-Donat}},
}

\bib{auslander1960}{article}{
      author={Auslander, Maurice},
      author={Goldman, Oscar},
       title={The Brauer Group of a Commutative Ring},
        date={1960},
        ISSN={0002-9947},
     journal={Transactions of the American Mathematical Society},
      volume={97},
       pages={367{\textendash}409},
         url={http://www.ams.org/mathscinet-getitem?mr=0121392},
}

\bib{azumaya1951}{article}{
      author={Azumaya, Gor\^{o}},
       title={On maximally central algebras},
        date={1951},
        ISSN={0027-7630},
     journal={Nagoya Mathematical Journal},
      volume={2},
       pages={119{\textendash}150},
         url={http://www.ams.org/mathscinet-getitem?mr=0040287},
}



\bib{gabber1981}{incollection}{
      author={Gabber, Ofer},
      editor={Kervaire, Michel},
      editor={Ojanguren, Manuel},
       title={Some theorems on azumaya algebras},
        date={1981},
   booktitle={Groupe de Brauer {(Sem.}, les {Plans-sur-Bex}, 1980)},
      series={Lecture Notes in Math.},
      volume={844},
   publisher={Springer},
     address={Berlin},
       pages={129{\textendash}209},
}

\bib{gille2006}{book}{
      author={Gille, Philippe},
      author={Szamuely, Tam\'{a}s},
       title={Central simple algebras and Galois cohomology},
      series={Cambridge Studies in Advanced Mathematics},
   publisher={Cambridge University Press},
     address={Cambridge},
        date={2006},
      volume={101},
        ISBN={978-0-521-86103-8; 0-521-86103-9},
}

\bib{giraud1971-a}{book}{
      author={Giraud, Jean},
       title={Cohomologie non ab\'{e}lienne},
   publisher={{Springer-Verlag}},
     address={Berlin},
        date={1971},
        note={Die Grundlehren der mathematischen Wissenschaften, Band 179},
}

\bib{grothendieck1968-a}{incollection}{
      author={Grothendieck, Alexander},
       title={Le groupe de brauer. i. alg\`{e}bres {d'Azumaya} et
  interpr\'{e}tations diverses},
        date={1968},
   booktitle={Dix expos\'{e}s sur la cohomologie des sch\'{e}mas},
   publisher={{North-Holland}},
     address={Amsterdam},
       pages={46{\textendash}66},
}

\extended{
\bib{heinloth2009}{article}{
	title = {The bigger Brauer group and twisted sheaves},
	volume = {322},
	issn = {0021-8693},
	number = {4},
	journal = {Journal of Algebra},
    author = {Heinloth, Jochen}
    author = {Schr\"{o}er, Stefan},
	month = {aug},
	year = {2009},
	pages = {1187--1195}
}
}

\bib{higgs1988}{article}{
      author={Higgs, R.~J.},
       title={On the degrees of projective representations},
        date={1988-05},
        ISSN={0017-0895},
     journal={Glasgow Mathematical Journal},
      volume={30},
      number={2},
       pages={133\ndash 135},
}


\extended{
\bib{karpilovsky}{book}{
    author={Karpilovsky, Gregory},
    title={Projective representations of finite groups},
    series={Monographs and Textbooks in Pure and Applied Mathematics},
    volume={94},
    publisher={Marcel Dekker Inc.},
    place={New York},
    date={1985},
    pages={xiii+644},
}

\bib{laumon2000}{book}{
	series = {Ergebnisse der Mathematik und ihrer Grenzgebiete},
	title = {Champs alg\'{e}briques},
	volume = {3. Folge, v. 39},
	isbn = {3540657614},
	publisher = {Springer},
	author = {Laumon, G\'{e}rard and {Moret-Bailly}, Laurent},
	month = {jan},
	year = {2000}
}


\bib{maclane1998}{book}{
      author={Mac~Lane, Saunders},
       title={Categories for the working mathematician},
     edition={Second},
      series={Graduate Texts in Mathematics},
   publisher={{Springer-Verlag}},
     address={New York},
        date={1998},
      volume={5},
        ISBN={0-387-98403-8},
}
}

\bib{maclane1992-a}{book}{
      author={Mac~Lane, Saunders},
      author={Moerdijk, Ieke},
       title={Sheaves in geometry and logic: a first introduction to topos
  theory},
      series={Universitext},
   publisher={{Springer-Verlag}},
     address={New York},
        date={1992},
        ISBN={0387977104},
}

\extended{
\bib{olsson2007}{article}{
	title = {Sheaves on Artin stacks},
	volume = {2007},
	issn = {0075-4102},
	number = {603},
	journal = {Journal f\"{u}r die reine und angewandte Mathematik {(Crelles} Journal)},
	author = {Olsson, Martin},
	month = {mar},
	year = {2007},
	pages = {55--112}
}
}

\end{biblist}
\end{bibdiv}
\end{document}